\documentclass[10pt, oneside]{amsart}
\usepackage[utf8]{inputenc}
\usepackage[english]{babel}
\usepackage[T1]{fontenc}
\usepackage{amsmath}
\usepackage{amsfonts}
\usepackage{amssymb}
\usepackage{amsthm}
\usepackage{amscd}
\usepackage{soul}
\usepackage{geometry}
\usepackage{enumitem} 
\usepackage{cancel} 
\usepackage{graphicx}
\usepackage{mathtools}
\usepackage{color}
\usepackage{comment}
\usepackage{yhmath}
\usepackage{color}
\definecolor{marin}{rgb}   {0.,   0.3,   0.7} 
\definecolor{rouge}{rgb}   {0.8,   0.,   0.} 
\definecolor{sepia}{rgb}   {0.8,   0.5,   0.} 
\usepackage[colorlinks,citecolor=marin,linkcolor=rouge,
            bookmarksopen,
            bookmarksnumbered
           ]{hyperref}

\newcommand\N{\mathbb{N}}
\newcommand\T{\mathbb{T}}
\newcommand\Z{\mathbb{Z}}

\newcommand\R{\mathbb{R}}
\newcommand\C{\mathbb{C}}

\newcommand{\dd}{\mathrm{d}}

\newcommand\eps{\varepsilon}

\newcommand\Id{\mathrm{Id}}

\newcommand\sinc{\mathrm{sinc}}

\newtheorem{theorem}{Theorem}
\newtheorem{proposition}{Proposition}

\newtheorem{lemma}{Lemma}

\newtheorem{conjecture}{Conjecture}

\begin{document}

\title{Continuum limit of the discrete nonlinear Klein-Gordon equation}

\author{Quentin Chauleur}
\address{INRIA Lille, Univ Lille \& Laboratoire Paul Painlevé,
CNRS UMR 8524 Lille, Cité Scientifique, 59655 Villeneuve-d'Ascq, France. }

\begin{abstract} 
We study the convergence of solutions of the discrete nonlinear Klein-Gordon equation on an infinite lattice in the continuum limit, using recent tools developed in the context of nonlinear discrete dispersive equations. Our approach relies in particular on the use of bilinear estimates of the Shannon interpolation alongside controls on the growth of discrete Sobolev norms of the solution. We conclude by giving perspectives on uniform dispersive estimates for nonlinear waves on lattices. 
\end{abstract}

\maketitle
We consider the discrete nonlinear Klein-Gordon equation
\begin{equation} \label{DNLKG} \tag{DNLKG}
\partial_t^2 u - \Delta_h u + u +  |u|^{p-1}u=0
\end{equation}
on a lattice $h\Z^d$ of step size $h>0$ with initial condition $u(0)=u_0$ and $\partial_t u(0)=u_1$. Here 
\[ \Delta_h u(a)=\sum_{j=1}^d \frac{u(a+h e_j)+ u(a-h e_j)-2 u(a)}{h^2}, \quad a \in h\Z^d,  \]
denotes the usual discrete Laplace operator which accounts for nearest neighbor interactions, with $(e_j)_{1\leq j \leq d}$ the canonical basis on $\R^d$. We restrict our attention to the dimensional cases $1 \leq d \leq 3$. 

Nonlinear waves in lattices \cite{kevrekidis2011} has received a lot of interest since the seminal investigation of Fermi,
Pasta and Ulam \cite{FPU1955} on the FPUT model alongside the development of soliton theory and integrable systems such as the Toda lattice \cite{toda1989} or the Ablowitz-Laddik equation \cite{ablowitz1975}. Discrete Klein-Gordon equations also naturally arise in the physics literature in order to describe Fluxon dynamics in one dimension parallel array of Josephson junctions \cite{ustinov1993}, or as a model for local denaturation of DNA \cite{peyrard1989}. Equation \eqref{DNLKG} enjoys the following energy conservation law
\begin{equation} \label{eq_energy_conservation}
E(t)\coloneqq \frac12 \| \partial_t u \|_{L^2_h}^2 +\frac12 \| \nabla_h u \|_{L^2_h}^2+ \frac12 \| u \|_{L^2_h}^2 + \frac{1}{p+1} \| u \|_{L^{p+1}_h}^{p+1} = E(0).
\end{equation}   
Our analysis will be performed on a particular set of parameters $p$ and $d$ satisfying
 \begin{equation} \label{param} \tag{param}
 \left\{
 \begin{aligned}
 & 1<p \quad & \text{for}  \ d=1,2,\\
 & 1<p<3 \quad & \text{for} \ d=3, 
 \end{aligned}
 \right.
 \end{equation}
which will imply from energy conservation the uniform bound $\| (u,\partial_t u ) \|_{H^1_h \times L^2_h} \leq C_{d}$ with respect both to the time $t \geq 0$ and the size of the lattice $h>0$, a feature that we will extensively use throughout this work. In particular, we will be interested in the limit $h \rightarrow 0$ of equation \eqref{DNLKG}, usually referred as the \textit{continuum limit}. From this perspective, equation \eqref{DNLKG} can also be seen as a finite difference scheme
for the numerical simulation of the well-known nonlinear Klein-Gordon equation
\begin{equation} \label{NLKG} \tag{NLKG}
\partial_t^2 \phi - \Delta \phi + \phi + |\phi|^{p-1} \phi=0,
\end{equation}
and one can reasonably ask in which framework and at which rate solutions of the discrete equation \eqref{DNLKG} converge to the ones of the continuous equation \eqref{NLKG}. We briefly recall that equation \eqref{NLKG} is a fundamental model in mathematical physics and has been extensively studied in a large amount of literatures.
It has been used as the equation of classical neutral scalar mesons, but also to study bosonic phases in massive stars in connection with Bose-Einstein condensation \cite{megias2022}. It also appears as a superfluid model to describe the creation and dynamics of quantized vortices in galaxies \cite{mauser2020}. We finally mention its link with nonlinear Schrödinger equations in the non-relativistic regime \cite{masmoudi2002,machihara2002}.

Other the past ten years, there have been many recent advances in the study of the continuum limit for dispersive equations since the work of Kirkpatrick, Lenzmann and Staffilani \cite{staffilani2013}, where the authors show the $L^2$ weak convergence of solutions of the discrete nonlinear Schrödinger equation (DNLS) in the continuum limit. The $L^2$ strong convergence of such solutions were then achieved by Hong and Yang in \cite{hong2019strong} alongside precise convergence rates in $h$, a result which was recently extended by Choi and Aceves \cite{choi2023} to the fractional nonlinear Schrödinger equation. However, to the best of the author's knowledge, no results are known for the continuum limit of the discrete nonlinear Klein Gordon, which is the main purpose of this paper.

This work follows from the conference \textsc{CJC-MA} which held in CentraleSupelec in September 2023, during which the author presented his recent study on discrete nonlinear Schrödinger equations \cite{chauleur2023dnls}, and announced that the strategy developed in this paper can be used for other dispersive equations. Note that contrary to this prior work, no uniform dispersive properties such as Strichartz estimates are used throughout the forthcoming proof, and comments about this feature will be developed at the end of this paper.

This paper is organized as follows. In section \ref{section_result}, we give some notations for discrete functional analysis and we state our main result Theorem \ref{theorem_continuum_limit}. Section \ref{section_growth} is devoted to the proof of uniform bounds on the growth of discrete Sobolev norms of solutions to \eqref{DNLKG}, and Theorem \ref{theorem_continuum_limit} is then proven in section \ref{section_continuum}. We conclude this work with section \ref{section_survey}, where we survey recent advances in nonlinear dispersive lattice equations and give perspectives.

\section{Discrete framework and main result} \label{section_result}

We denote respectively by $L^p(h\Z^d)$, for $1 \leq p < \infty$, and $L^{\infty}(h \Z^d)$ (or sometimes more compactly $L^p_h$ and $L^{\infty}_h$ in mathematical mode) the discrete Lebesgue spaces induced by the norms
\[ \| g \|_{L^p_h}^p = h^d \sum_{a \in h\Z^d} |g(a)|^p \quad \text{and} \quad \| g \|_{L^{\infty}_h} = \sup_{a \in h\Z^d} |g(a)|  . \]
One can also define the \textit{forward discrete gradient} 
\[ \nabla_{h,j}^+ g(a)=\frac{g(a+he_j)-g(a)}{h}, \quad \nabla_h^+ = \left( \nabla_{h,1}^+,\ldots,\nabla_{h,d}^+  \right)^{\top} \]
for any $a \in h \Z^d$, as well as the \textit{discrete gradient}
\[ \nabla_{h,j} g(a)=\frac{g(a+he_j)-g(a-h e_j)}{h}, \quad \nabla_h = \left( \nabla_{h,1},\ldots,\nabla_{h,d}  \right)^{\top}. \]
The discrete Fourier transform of a function $g \in L^2(h \Z^d)$ and its inversion formula are  given by
\[ \widehat{g} (\xi) = h^d \sum_{a \in h \Z^d} g(a) e^{-ia \cdot \xi} \quad \text{and} \quad g(a)= \frac{1}{(2 \pi)^d}\int_{\T_h^d} \widehat{g}(\xi) e^{ia \cdot \xi} \dd \xi \]
for $ \xi \in \T_h^d=\R^d / \left(\frac{2 \pi}{h} \Z^d \right)$ and $a \in h\Z^d$, defining an isometry from $L^2(h\Z^d)$ to $L^2(\T_h^d)$. One can also define  discrete Sobolev spaces $H^s(h\Z^d)$ (or $H^s_h$) for any $s \in \R$ with the norm
\[ \| g \|^2_{H^s_h} =  \frac{1}{(2 \pi) ^d} \int_{\T_h^d} \left( 1+\frac{4}{h^2} \sum_{j=1}^d \sin  \left( \frac{h \xi_j }{2}  \right)^2 \right)^s \left| \widehat{u}( \xi ) \right|^2 \dd \xi. \]
In order to compare discrete and continuous functions, we need to define a projection operator on the grid and an interpolation operator to lift a discrete function on the continuous space. We denote the \textit{mean projection}
\[ \pi_h \varphi (a) =  \frac{1}{h^d} \int_{a+ \left[ -\frac{h}{2}, \frac{h}{2} \right[^d} \varphi(x) \dd x \]
for all $a \in h \Z^d$. We also define the \textit{Shannon interpolation} of a function $u : L^2(h\Z^d) \rightarrow \C$ by
\[    \mathcal{S}_h u \coloneqq \mathcal{F}^{-1} \left( \mathbf{1}_{\T_h^d} \widehat{u}  \right),  \]
which allows to extend a discrete function into a real function whose Fourier transform is compactly supported in $\T^d_h$, and where $\mathcal{F}$ defined by 
\[ \mathcal{F} f(\xi) = \int_{\R^d} f(x)e^{-ix\cdot \xi} \dd x  \]
for all $\xi \in \R^d$ denotes the usual Fourier transform on $\R^d$. We now state our main result:

\begin{theorem} \label{theorem_continuum_limit}
Let $s \in \N^*$ with $(p,d)$ satisfying \eqref{param}. Let $\phi \in \mathcal{C}(\R;H^{s+2}(\R^d))$ be the unique solution of \eqref{NLKG} with initial condition $(\phi_0,\phi_1) \in H^{s+2}(\R^d) \times H^{s+1}(\R^d)$, and let $u$ be the unique solution of \eqref{DNLKG} with initial condition $(u_0,u_1)=(\pi_h \phi_0,\pi_h \phi_1)$. Then
\[ \| \mathcal{S}_h u(t) - \phi(t) \|_{H^s(\R^d)} \leq C h e^{B(1+t)^{(p-1)}}   \| \mathcal{S}_h u_0 - \phi_0 \|_{H^s(\R^d)}, \]
where $B$ and $C$ are constants depending on $d$, $p$, $s$ and $\|(\phi_0,\phi_1) \|_{H^{s+2}(\R^d) \times H^{s+1}(\R^d)}$.	
\end{theorem}

\section{Growth of discrete Sobolev norms} \label{section_growth}

In this section we adapt the strategy of Pampu \cite{pampu2021} to infer upper bounds on the growth in time of Sobolev norms for the nonlinear Klein-Gordon equation to our discrete setting. This proof relies on the use of \textit{modified energies}, a strategy that have proved useful in various contexts and which have very little dependence on the underlying geometry of the problem, making it very appealing from the discrete point of view. Note that this strategy was already used by the author for the discrete nonlinear Schrödinger \cite{chauleur2023dnls}, inspired by the work of Planchon, Tzvetkov and Visciglia \cite{planchon2017}.

Let $T>0$, and $(u,\partial_t u) \in \mathcal{C}(\left[0,T \right];H^2_h \times H^1_h)$ be the solution of equation \eqref{DNLKG} with initial condition $(u_0,u_1) \in H^2_h \times H^1_h$. We give all the following proves in dimension $d=2$ with $1<p< \infty$, but the same arguments work in the three dimensional case $d=3$ (which is actually performed in \cite{pampu2021}) with nonlinearity $1<p<3$  relying on the Sobolev embedding $H^1_h \subset L^q_h$ for all $1\leq q\leq6$. Note also that the proof is a bit simpler in the case $d=1$, as we have $H^1_h \subset L^{\infty}_h$.

 In view of equation \eqref{DNLKG} and conservation of energy \eqref{eq_energy_conservation} alongside discrete Sobolev embeddings, we see that $\partial_t^2 u  \in \mathcal{C}(\left[0,T \right];L^2_h)$ and that
\begin{equation} \label{eq_partial_t_Delta}
\| \partial_t^2 u - \Delta u  \|_{L^2_h} \leq C
\end{equation}
where $C=C(d,E(0))>0$. We now define the modified energy
\begin{equation} \label{eq_modified_energy}
\mathcal{E}(t) \coloneqq \frac12 \left( \| \partial_t^2 u(t) \|^2_{L^2_h} + \| \nabla_h^+  \partial_t u(t) \|^2_{L^2_h} + \| \partial_t u(t) \|^2_{L^2_h} \right)
\end{equation} 
for all $t \in  \left[0,T \right]$. In particular differentiating with respect to time we infer that
\[ \frac{\dd}{\dd t} \mathcal{E}(t) = - \langle \partial_t(|u|^{p-1} u), \partial_t^2 u \rangle_h.   \]
Now integrating in time between $0$ and $T$, we get by Hölder's inequality that
\[ \mathcal{E}(T) - \mathcal{E}(0) \leq \int_0^T \sum_{h\Z^d} |u|^{p-1} | \partial_t u| | \partial_t^2 u| \leq \int_0^T \| u^{p-1}\|_{L^{\frac{2(2+\eta)}{\eta}}_h} \| \partial_t u \|_{L^{2+\eta}_h} \| \partial_t^2 u \|_{L^2_h}   \]
for any small $\eta>0$ yet to be fixed. Then from interpolation in discrete Lebesgue spaces alongside discrete Sobolev embedding $H^1_h \subset L^q_h$ for all~$q>1$, we get that
\[\mathcal{E}(T) - \mathcal{E}(0)  \leq \int_0^T \| u\|_{L^{\frac{2(2+\eta)(p-1)}{\eta}}_h}^{p-1} \| \partial_t u \|_{L^2_h}^{\frac{2-\eta^2}{2+\eta}} \| \partial_t u \|_{L^{1+\frac{1}{\eta}}_h}^{\frac{\eta(\eta+1)}{2+\eta}}   \| \partial_t^2 u \|_{L^2_h}  \lesssim T \| \partial_t u \|_{L^{\infty}_T H^1_h}^{\frac{\eta(\eta+1)}{2+\eta}}  \| \partial_t^2 u \|_{L^{\infty}_T L^2_h}, \] 
so denoting $\eps = 2\eta(\eta+1)/(2+\eta)>0$ as small as needed, we infer in view of \eqref{eq_modified_energy} that
\begin{equation} \label{eq_bound_mathcal_E}
\mathcal{E}(T) - \mathcal{E}(0) \lesssim T \sup_{t \in \left[0,T \right]} \mathcal{E}(t)^{\frac{1+ \eps}{2}} .
\end{equation}
This last equation will be useful to prove the following bound on the growth of the $H^2_h \times H^1_h$-norm of $(u,\partial_t u)$, using the equivalence between the $H^2_h \times H^1_h$ norm of $(u,\partial_t u)$ and $\mathcal{E}(t)$.

\begin{proposition} \label{prop_bound_H2}
Let $(u,\partial_t u)$ be solution to \eqref{DNLKG} with $(u_0,u_1) \in H^2_h \times H^1_h$, then for all $\eps>0$,
\[ \sup_{t \in \left[0,T\right] } \|(u,\partial_t u)(t)\|_{H^2_h \times H^1_h} \lesssim 
\left\{ \begin{aligned}
& (1+T) \quad &\text{if} \ d=1, \\
&(1+T)^{\frac{1}{1-\eps}} \quad &\text{if} \ d=2,\\
&(1+T)^{\frac{2}{3-p}} \quad & \text{if} \ d=3.
\end{aligned} \right.
 \]
\end{proposition}
\begin{proof}
As before, we restrain the proof to the two-dimensional case $d=2$, the cases $d=1$ and $d=3$ (with $1<p<3$) being proven similarly. Let first write that thanks to equation \eqref{eq_partial_t_Delta},
\[ \mathcal{E}(t) \leq \| \partial_t^2 u(t) - \Delta u(t) \|_{L^2_h}^2 + \| (u,\partial_t u)(t) \|_{H^2_h \times H^1_h}^2 \leq C + \| (u,\partial_t u)(t) \|_{H^2_h \times H^1_h}^2.  \]
On the other hand, for all $\tau \in (0,1)$ and small  $\eta>0$, we infer thanks to equation \eqref{eq_bound_mathcal_E} that
\begin{align*}
\| (u,\partial_t u )(\tau) \|_{H^2_h \times H^1_h}^2 & \lesssim \| (u,\partial_t u )(\tau) \|_{H^1_h \times L^2_h}^2 + \| \Delta_h u(\tau) \|_{L^2_h}^2 + \| \nabla_h^+ \partial_t u(\tau) \|_{L^2_h}^2 \\
& \leq C + 2 \| \partial_t u(\tau) - \Delta u(\tau) \|_{L^2_h}^2 +2 \mathcal{E}(\tau) \\
& \leq C+ 2 \mathcal{E}(0)+C \tau \sup_{t\in\left[0,\tau \right]} \mathcal{E}(t)^{\frac{1+\eps}{2}} \\
& \leq C+ 2 \| (u,\partial_t u )(0) \|_{H^2_h \times H^1_h}^2 + \tilde{C} \tau \sup_{t\in\left[0,\tau \right]}\| (u,\partial_t u )(t) \|_{H^2_h \times H^1_h}^{1 + \eps},
\end{align*}
where the constants $C$ and $\tilde{C}$ are independents of $\tau$. In particular, for $0<\tau \leq \tau_0$ small enough such that $\tilde{C} \tau_0<1$, we get from previous estimate that
\[ \sup_{t \in \left[0,\tau_0 \right]} \| (u,\partial_t u )(t) \|_{H^2_h \times H^1_h}^2 \leq 2 \| (u,\partial_t u )(0) \|_{H^2_h \times H^1_h}^2 + C, \]
which gives that
\[ \| (u,\partial_t u )(\tau_0) \|_{H^2_h \times H^1_h}^2 \leq 2 \| (u,\partial_t u )(0) \|_{H^2_h \times H^1_h}^2 + C \left(1+ \| (u(0),\partial_t u(0)) \|_{H^2_h \times H^1_h}^2 \right)^{\frac{1+\eps}{2}}. \]
We know remark that denoting $\alpha_n=1+ \| (u,\partial_t u )(n\tau_0) \|_{H^2_h \times H^1_h}^2$, the sequence $(\alpha_n)_{n\geq 0}$ satisfies that $ \alpha_{n+1} \leq 2 \alpha_n + C\alpha_n^{1-\frac{1-\eta}{2}}$, which leads by induction that $\alpha_n \leq C n^{\frac{2}{1-\eps}}$, or rewriting in terms of $(u,\partial_t u)$ that
\[ \sup_{t \in \left[ 0,T \right]} \| (u,\partial_t u)(t) \|_{H^2_h \times H^1_h} \leq C (1+T)^{\frac{1}{1-\eps}} \]
which ends the proof for $d=2$.
\end{proof}

To estimate higher discrete Sobolev norms $H^{k+1}_h\times H^k_h$ for $k \geq 2$, one can define the \textit{higher modified energies}
\[ \mathcal{E}_k(t)  \coloneqq \frac12 \left( \| \partial_t^{k+1} u(t) \|^2_{L^2_h} + \| \nabla_h^+  \partial_t^k u(t) \|^2_{L^2_h} + \| \partial_t^k u(t) \|^2_{L^2_h} \right) \]
which enjoys the same way that
\[ \frac{\dd}{\dd t} \mathcal{E}_k(t) = - \langle \partial_t^k(|u|^{p-1} u), \partial_t^{k+1} u \rangle_h.   \]
However, as in proof of \cite[Proposition 5]{pampu2021}, one can show by induction on $k$ that for any dimension $1\leq d \leq 3$ (with $1<p<3$ if $d=3$), 
\[ \mathcal{E}_k(T) - \mathcal{E}_k(0) \lesssim T \sup_{t \in \left[0,T \right]} \mathcal{E}_k(t)^{\frac12} ,  \]
which is a better estimate than equation \eqref{eq_bound_mathcal_E} and leads to the following result:
\begin{proposition} \label{prop_bound_Hk}
Let $(u,\partial_t u)$ be solution to \eqref{DNLKG} with $(u_0,u_1) \in H^{k+1}_h \times H^k_h$ for $k \geq 2$, then
\[ \sup_{t \in \left[0,T\right] } \|(u,\partial_t u)(t)\|_{H^{k+1}_h \times H^k_h} \lesssim (1+T). \]
\end{proposition}
\begin{proof}
The proof follows the exact same lines as the one given through \cite[Section 4]{pampu2021} as no dispsersive estimates are used throughout the proof. 
\end{proof}

\section{Strong convergence in the continuum limit} \label{section_continuum}
From Duhamel's formula, denoting $K_h(t)=\frac{\sin \left( t \sqrt{1-\Delta_h}\right)}{\sqrt{1-\Delta_h}}$ and $\dot{K}_h (t)=\cos ( t \sqrt{1-\Delta_h})$, we write
\[  u(t)= \dot{K}_h (t) u_0 - K_h(t) u_1 - \int_0^t  K_h(t-\tau) \left( |u|^{p-1} u \right)(\tau) \dd \tau \]
for any  $u$ solution of \eqref{DNLKG}, and analogous formulas hold for $\dot{K}(t)$, $K(t)$ and $\phi(t)$ solution of \eqref{NLKG} following the notations from \cite{ginibre1985,ginibre1989}. We will decompose our analysis on the following integrals
\begin{equation} \label{eq_duhamel_diff}
\begin{aligned} 
\| \mathcal{S}_h u (t) - \phi(t) \|_{H^s(\R^d)} & \leq \| \mathcal{S}_h \dot{K}_h (t) u_0 - \dot{K} (t) \phi_0  \|_{H^s(\R^d)}  + \| \mathcal{S}_h K_h(t) u_1 - K(t) \phi_0  \|_{H^s(\R^d)}\\
&  +  \int_0^t \left\| \left( \mathcal{S}_h K_h(t) - K(t) \mathcal{S}_h \right) \left( |u|^{p-1} u \right)(\tau) \right\|_{H^s(\R^d)} \dd \tau \\
 	& +\int_0^t \left\| \mathcal{S}_h \left( |u|^{p-1} u \right)(\tau) -  \left( \left| \mathcal{S}_h u \right|^{p-1} \mathcal{S}_h u \right)(\tau)  \right\|_{H^s(\R^d)} \dd \tau \\
 	& +  \int_0^t \left\| \left( \left| \mathcal{S}_h u \right|^{p-1} \mathcal{S}_h u \right)(\tau) -  \left(|\phi|^{p-1} \phi \right)(\tau) \right\|_{H^s(\R^d)} \dd \tau \\
 	& =: J_1(t) + J_2(t) + J_3(t) + J_4(t) + J_5(t).
 \end{aligned}
\end{equation}

\subsection{The linear flow}

It is quite direct (see for instance \cite[Lemma 5.1]{hong2019strong}) that for $\varphi \in H^s(\R^d)$, $\| \pi_h \varphi \|_{H^s_h} \lesssim \| \varphi \|_{H^s}$, which is known as the \textit{boundedness} property of the discretization $\pi_h$. We now deal with the error made by consequently projecting then interpolating a function $\varphi$, as it may not be explicitly written in the literature.

\begin{lemma} \label{lemma_proj_inter}
Let $s \geq 0$ and $\varphi \in L^2(\R^d)$, then for all $\xi \in \T_h^d$,
\[ \widehat{ \pi_h \varphi} (\xi) = \mathcal{F} \varphi (\xi) \prod_{j=1}^d \sinc \left(\frac{h \xi_j}{2} \right).  \]
\end{lemma}
\begin{proof}
From inverse Fourier transform property we compute that for all $a \in h\Z^d$,
\[ \pi_h \varphi (a) =  \frac{1}{(2\pi h)^d} \int_{a+ \left[ -\frac{h}{2}, \frac{h}{2} \right[^d} \int_{\R^d} e^{i x \cdot \xi} \mathcal{F} \varphi (\xi) \dd \xi \dd x = \frac{1}{(2 \pi)^d} \int_{\R^d} e^{ia \cdot \xi} \mathcal{F} \varphi (\xi) \prod_{j=1}^d \sinc \left(h \xi_j/2 \right)  \dd \xi.  \]
On the other hand we know that 
\[  \Pi_h f (a)=  \frac{1}{(2\pi)^d} \int_{\T_h^d} e^{ia \cdot \xi}  \widehat{\Pi_h f} (\xi) \dd \xi, \]
 which gives the result.
\end{proof}

Estimates on $J_1$ and $J_2$ will be a direct consequence of the following statement:

\begin{proposition} \label{prop_linear_flow}
Let $u_0 = \pi_h \phi_0$ and $u_1 = \pi_h \varphi_1$, then for all $s\geq 0$, 
\[ \| \mathcal{S}_h \dot{K}_h (t) u_0 - \dot{K} (t) \phi_0  \|_{H^s} \leq C h (1+t)\| \phi_0 \|_{H^{s+2}} \quad \text{and} \quad \| \mathcal{S}_h K_h (t) u_1 - K (t) \phi_1  \|_{H^s} \leq C h (1+t) \| \phi_1 \|_{H^{s+2}}. \]
\end{proposition}
\begin{proof}
We focus on the first estimate, as the second one can be proved similarly. We write that
\[  \| \mathcal{S}_h \dot{K}_h (t) u_0 - \dot{K} (t) \phi_0  \|_{H^s}  \leq  \| \mathcal{S}_h \dot{K}_h (t) u_0 - \dot{K} (t) \mathcal{S}_h u_0  \|_{H^s} + \|\dot{K} (t)  (\mathcal{S}_h u_0 - \phi_0 )\|_{H^s} \eqqcolon I_1(t)+I_2(t). \]
As $\cos$ and $\sqrt{1+\cdot}$ are 1-Lipschitz and as $|\frac{4}{h^2} \sum_{j=1}^d \sin (h \xi_j/2)^2 - |\xi|^2| \lesssim h^2 \xi^4$ for $\xi \in \T^d_h$, 
\begin{align*}
 I_1(t)^2 & = \int_{\T_h^d} (1+|\xi|^2)^s \left|\cos \left(t \sqrt{1+ \frac{4}{h^2} \sum_{j=1}^d  \sin \left( \frac{h\xi_j}{2} \right)^2} \right) - \cos(t \sqrt{1+|\xi|^2}) \right|^2 |\widehat{u}_0(\xi)|^2 \dd \xi \\
 & \lesssim t^2 h^2 \int_{|\xi|\leq \frac{\pi}{\sqrt{h}}} (1+|\xi|^2)^s |\xi|^4 \left| \widehat{u_0}(\xi) \right|^2 \dd \xi + \int_{\T^d_h \cap \left\{|\xi|> \frac{\pi}{\sqrt{h}}\right\} } \frac{(1+|\xi|^2)^{s+2}}{(1+|\xi|^2)^2 }\left| \widehat{u_0}(\xi) \right|^2 \dd \xi \\
& \lesssim t^2 h^2 \int_{|\xi|\leq \frac{\pi}{\sqrt{h}}} (1+|\xi|^2)^{s+2} \left| \widehat{u_0}(\xi) \right|^2 \dd \xi + h^2 \int_{\T^d_h \cap \left\{|\xi|> \frac{\pi}{\sqrt{h}}\right\} } (1+|\xi|^2)^{s+2}\left| \widehat{u_0}(\xi) \right|^2 \dd \xi,
 \end{align*}
 which gives the first bound as $\| u_0 \|_{H^{s+2}_h} \lesssim \| \phi_0\|_{H^{s+2}}$. Now note that from Lemma \ref{lemma_proj_inter},
\begin{multline*}
  \| \mathcal{S}_h \circ \pi_h \varphi - \varphi \|_{H^s}^2  = \int_{\T^d_h} \left(1 + |\xi|^2 \right)^s \left| \widehat{ \pi_h f} (\xi) - \mathcal{F}f (\xi)  \right|^2 \dd \xi + \int_{\R^d \backslash \T^d_h} \left(1 + |\xi|^2 \right)^s \left| \mathcal{F}f (\xi)  \right|^2 \dd \xi \\
   \lesssim \int_{\T^d_h} \left(1 + |\xi|^2 \right)^s \left| \mathcal{F}f (\xi)  \right|^2 \left| \prod_{j=1}^d \sinc \left( h \xi_j/2 \right) -1 \right|^2 \dd \xi + \int_{\R^d \backslash \T^d_h} \frac{\left(1 + |\xi|^2 \right)^{s+2}}{(1 + |\xi|^2)^2} \left| \mathcal{F}f (\xi)  \right|^2 \dd \xi \\
   \lesssim h^4 \int_{\T^d_h} \left(1 + |\xi|^2 \right)^s |\xi|^4 \left| \mathcal{F}f (\xi)  \right|^2 \dd \xi + h^4 \int_{\R^d \backslash \T^d_h} \left(1 + |\xi|^2 \right)^{s+2} \left| \mathcal{F} \varphi(\xi)  \right|^2 \dd \xi,
 \end{multline*}
as $| \prod_{j=1}^d \sinc ( h \xi_j /2) -1 | \lesssim h^2 |\xi|^2$ so for $I_2$ we directly get that
 \[  I_2(t) \leq \| \mathcal{S}_h u_0 - \phi_0 \|_{H^s} \lesssim h^2 \|  \phi_0 \|_{H^{s+2}},\]
which gives the result as $h \leq 1$.
\end{proof}

\subsection{Linear flow on the nonlinearity}

To estimate $J_3$, we temporary denote $F=|u(\tau)|^{p-1} u(\tau)$, so that
\[  J_3(t) \leq \int_0^t \left( \| \mathcal{S}_h K_h(\tau) (F-\pi_h \mathcal{S}_h F) \|_{H^s} +  \| (\mathcal{S}_h K_h(\tau) \pi_h - K(\tau) ) \mathcal{S}_h F \|_{H^s} \right) \dd \tau \eqqcolon I_3(t) + I_4(t) .  \]
A bound on $I_4$ can be directly derived using Proposition \ref{prop_linear_flow}, so that
\[ I_4(t) \lesssim h \int_0^t (1+\tau) \| F \|_{H^{s+2}_h} \dd \tau \lesssim h \int_0^t (1+\tau)\| u\|_{H^{s+2}_h}^p \dd \tau \]
as $s+2 > d/2$ since $1 \leq d \leq 3$. On the other hand, as $\| \mathcal{S}_h \varphi \|_{H^s} \lesssim \| \varphi \|_{H^s_h}$ (see \cite[Lemma 5]{chauleur2023dnls}),
\[ I_3(t) \lesssim \int_0^t \| K_h(\tau) (F - \pi_h \mathcal{S}_h F) \|_{H^s_h} \dd \tau \lesssim \int_0^t \| F -\pi_h \mathcal{S}_h F \|_{H^s_h} \dd \tau,  \]
and from Lemma \ref{lemma_proj_inter} we know that $\widehat{\pi_h \mathcal{S}_h u} (\xi)= \widehat{u}(\xi) \prod_{j=1}^d \sinc(h \xi_j/2)$ so by Jensen inequality
\begin{align*}
I_3(t)^2 & \lesssim \frac{1}{(2\pi)^d} \int_0^t \int_{\T_h^d} \left( 1+\frac{4}{h^2} \sum_{j=1}^d \sin  \left( \frac{h \xi_j }{2}  \right)^2 \right)^s \left| \widehat{F}( \xi ) \right|^2 \left|1 - \prod_{j=1}^d \sinc(h\xi_j/2)  \right|^2 \dd \xi \dd \tau \\
& \lesssim h^4 \int_0^t \| F \|_{H^{s+2}_h}^2 \dd \tau \lesssim h^4 \int_0^t \| u\|_{H^{s+2}_h}^{2p} \dd \tau.
\end{align*}  
as $|\xi|^2 \leq \frac{\pi^2}{4} \sum_{j=1}^d \frac{4}{h^2} \sin(h \xi_j/2)^2$ for all $\xi \in \T_h^d$. Combining the bounds on $I_3$ and $I_4$, we get that
\begin{equation} \label{eq_bound_J3}
J_3(t) \lesssim h \int_0^t (1+\tau) \| u(\tau) \|_{H^{s+2}_h}^p \dd \tau \lesssim h (1+t)^{(p+1)}
\end{equation}
using Proposition \ref{prop_bound_Hk} since $s \in \N^*$ so $s+2 \geq 3$.

\subsection{Aliasing and Gronwall argument}
A straightforward bound on $J_4$ can be derive the exact same way as in \cite[Section 5.2]{chauleur2023dnls}, which relies on the bilinear estimate satisfies by the Shannon interpolation \cite[Proposition 6]{chauleur2023dnls} stating that for $f$, $g \in H^{\delta}(h\Z^d)$ with $\delta>\frac{d}{2}$ and for $0 \leq s \leq \delta$, we have
\[  \|  \mathcal{S}_h (fg) - \mathcal{S}_h f \mathcal{S}_h g \|_{H^s} \lesssim h^{\delta-s} \| \mathcal{S}_h f \|_{H^{\delta}} \| \mathcal{S}_h g \|_{H^{\delta}}. \]
This directly gives in our setting (using Proposition \ref{prop_bound_Hk}) that
\begin{equation} \label{eq_bound_J4}
 J_4(t) \lesssim h^2 \int_0^t \| \mathcal{S}_h u(\tau) \|_{H^{s+2}}^p \dd \tau \lesssim h^2 \int_0^t\| u(\tau) \|_{H^{s+2}_h}^p \dd \tau \lesssim h^2 (1+t)^{2}.
\end{equation}
Turning now to $J_5$,  using the well-known identity $\left| |f|^{p-1}f - |g|^{p-1}g \right| \lesssim_p(|f|+|g|)^{p-1} |f-g|$ we can write that for all $\tau \in (0,t)$,
\[ \left\| \left(\left|\mathcal{S}_h u \right|^{p-1} \mathcal{S}_h u \right)(\tau) -  \left(|\psi|^{p-1} \psi \right)(\tau) \right\|_{H^s} \lesssim  \left( \| \mathcal{S}_h u(\tau) \|_{H^{s+2}}^{p-1}+\| \phi(\tau) \|_{H^{s+2}}^{p-1} \right) \| \mathcal{S}_h u(\tau) - \phi(\tau) \|_{H^s} \]
as $s+2 >d/2$, so 
\begin{equation} \label{eq_bound_J5}
 J_4(t) \lesssim \int_0^t (1+ \tau)^{(p-1)} \| \mathcal{S}_h u(\tau) - \phi(\tau) \|_{H^s} \dd \tau.
\end{equation}
Using respectively the bounds on $J_1$ and $J_2$ from Proposition \ref{prop_linear_flow}, as well as the ones on $J_3$ \eqref{eq_bound_J3}, on $J_4$ \eqref{eq_bound_J4} and on $J_5$ \eqref{eq_bound_J5} in estimate \eqref{eq_duhamel_diff}, we get the result of Theorem \ref{theorem_continuum_limit} by Gronwall lemma. Note that the polynomial terms $(1+t)^{\alpha}$ for any $\alpha>0$ can absorbed by the exponentially growing term $e^{B (1+t)^{(p-1)}}$ taking a larger constant $B>0$.

\section{Decay estimates for discrete wave equations} \label{section_survey}

One may wonder if it would be possible to obtain continuum limit properties such has Theorem~\ref{theorem_continuum_limit} for more general cases, for instance in the cubic case $p=3$ in the three-dimensional case $d=3$, which appears as a limiting case in Proposition \ref{prop_bound_H2}. In order to follow the proof of \cite{pampu2021}, this would require the use of Strichartz estimates so one can trade space regularity for time integrability in Hölder's inequality application to estimate equation \eqref{eq_bound_mathcal_E}. It is now a standard argument since the work of Keel and Tao \cite{tao1998} that Strichartz estimates are deduced from a bounded $L^2$-estimate of the associated linear semi-group (usually derived from conservation laws) alongside a time decay estimate of the $L^{\infty}$-norm of such linear flow. In fact, the derivation of decay estimates for discrete dispersive equations has drawn some attention through the last decades, and we survey some of these results in a unified framework.

\subsection{Dispersive estimates on lattices} \label{subsection_decay}

Consider a discrete dispersion relation of the form 
\[ \omega_h(\xi) = \left( m^2 + \frac{4}{h^2} \sum_{j=1}^d \sin  \left( \frac{h \xi_j }{2}  \right)^2   \right)^{\frac{\alpha}{2}} \]
associated to the linear flow $U_h(t)=\exp (-it (m^2-\Delta_h)^{\frac{\alpha}{2}})$. This covers several interesting models:
\begin{itemize}
	\item The case $m=0$ and $\alpha=1$ corresponds to the \textbf{discrete wave equation}.
	\item The case $m=0$ and $\alpha=2$ corresponds to the \textbf{discrete Schrödinger equation}.
	\item The case $m=1$ and $\alpha=1$ corresponds to the \textbf{discrete Klein-Gordon equation}.
\end{itemize}
The evolution of an initial state $\varphi$ of discrete Fourier transform $\eta$ under the linear flow $U_h(t)$ is given by the oscillatory integral
\begin{equation} \label{eq_oscillatory_integral}
 I(t,a,\eta) = \int_{\T^d_h} e^{i(a \cdot \xi - t \omega_h(\xi))} \eta(\xi) \dd \xi. 
 \end{equation}
for any $a \in h\Z^d$. In the case of the Schrödinger flow, as performed in the seminal work of Stefanov and Kevrekidis \cite{kevrekidis2005}, one can separate variables in order to reduce the problem to the $d=1$ case, and an application of the well-known Van der Corput lemma gives that
\begin{equation} \label{eq_decay_Schrödinger}
\| e^{i t \Delta_h} \varphi \|_{L^{\infty}(h\Z^d)} \leq \frac{C}{|th|^{\frac{d}{3}}}  \| \varphi \|_{L^1(h\Z^d)}	
\end{equation}
for any initial data $\varphi \in L^1(h\Z^d)$, which is proven to be sharp \cite{ignat2007}. Such a dispersive estimate has to be compared with the usual decay in $t^{-\frac{d}{2}}$ in the continuous free case on $\R^d$. It displays a weaker dispersion estimate than the continuous case, a pathological behavior induced by critical points and a lack of convexity of the symbol of the discrete operator $\Delta_h$. 

From now on we fix $h=1$ to introduce the forthcoming results, and we denote $\ell^{\infty}(\Z^d) \coloneqq L^{\infty}(h\Z^d)$ in order to match the notations of these papers. One-dimensional decay of the discrete Klein-Gordon flow was also derived in \cite{kevrekidis2005} from the same techniques, which can be written as
\[  \| e^{-i t \sqrt{1-\Delta_h}} \varphi \|_{\ell^{\infty}(\Z)} \lesssim \frac{1}{|t|^{\frac{1}{3}}}  \| \varphi \|_{L^1(\Z)}.	  \]
However, wave dispersion relations (for $\alpha \neq 2$) fail to have this separation-of-variables property, leading to more involved analysis as dimension increases. Back in 1998, Schultz proved in his breakthrough work \cite{schultz1998} that for the discrete wave equation in dimensions $d=2,3$,
\[  \left| \
\begin{aligned}
&\| e^{-it\sqrt{-\Delta_h}} \|_{\ell^{\infty}(\Z^2)} \leq C (1+|t|)^{-\frac{2}{3}}, \\
&\| e^{-it\sqrt{-\Delta_h}} \|_{\ell^{\infty}(\Z^3)} \leq C (1+|t|)^{-\frac{7}{6}} ,
\end{aligned} \right.    \]
where $C=C(\eta)>0$, a result that have only been very recently extended to the fourth dimensional case $d=4$ in \cite{cheng2023} which states the sharp bound 
\[ \| e^{-it\sqrt{-\Delta_h}} \|_{\ell^{\infty}(\Z^4)} \leq C (1+|t|)^{-\frac{3}{2}}\log(2+ |t|),  \]
where the proof relies on the analysis of Newton polyhedra. For the discrete Klein-Gordon equation, Borovyk and Goldberg proved in \cite{borovyk2017} that in dimension $d=2$,
\[ \| e^{-it\sqrt{1-\Delta_h}} \|_{\ell^{\infty}(\Z^2)} \leq C (1+|t|)^{-\frac{3}{4}},  \]
a result which was extended a few years later by Cuenin and Ikromov \cite{cuenin2021}
\[ \left| \
\begin{aligned}
&\| e^{-it\sqrt{1-\Delta_h}} \|_{\ell^{\infty}(\Z^3)} \leq C (1+|t|)^{-\frac{7}{6}}, \\
&\| e^{-it\sqrt{1-\Delta_h}} \|_{\ell^{\infty}(\Z^4)} \leq C (1+|t|)^{-\frac{3}{2}}\log(2+ |t|) ,
\end{aligned} \right.   \]
so that in dimension 3 and 4 the discrete Klein-Gordon and wave equations share the same decay rates, whereas in the two-dimensional setting $d=2$ the dispersion of the wave flow is a bit slower than the one of the Klein-Gordon equation. 

For higher dimensions $d\geq 5$, the Klein-Gordon flow is conjectured in \cite{cuenin2021} to decay as
\[ \| e^{-it\sqrt{1-\Delta_h}} \|_{\ell^{\infty}(\Z^d)} \leq C (1+|t|)^{-\frac{2d+1}{6}}\log(2+ |t|)^{d-4},  \]
whereas the wave flow is conjectured in \cite{cheng2023} to behave like 
\[ \| e^{-it\sqrt{\Delta_h}} \|_{\ell^{\infty}(\Z^d)} \leq C (1+|t|)^{-\frac{2d+1}{6}}.  \]

\subsection{Uniform Strichartz estimates for the discrete Klein-Gordon equation}

Decay estimates of Section \ref{subsection_decay} are inherently not uniform with respect to the mesh size $h>0$, making them a priori useless in the study of the continuum limits of such discrete systems as $h \rightarrow 0$.

Several works of Ignat and Zuazua for the discrete Schrödinger equation \cite{ignat2009,ignat2012} and Audiard \cite{audiard2013} for the discrete critical Korteweg-de Vries use Fourier filtering methods alongside two-grid algorithm in order to remove the bad behavior frequencies from the discrete operator $\Delta_h$. This strategy allows to recover modified Strichartz estimates which are uniform with respect to $h$, that can be described as follows in the Schrödinger setting: denote by $\widetilde{\Pi}_h$ the piecewise linear extension operator $\widetilde{\Pi}_h : L^2(4h\Z^d) \rightarrow L^2(h\Z^d)$, then one recovers that
\[ \| e^{it \Delta_h} \widetilde{\Pi}_h \varphi^{4h} \|_{L^q_t(\R;L^r(h\Z^d))} \leq C(d,r)   \| \widetilde{\Pi}_h \varphi^{4h} \|_{L^2(h\Z^d)} \]
for all $\varphi^{4h} \in L^2(4h\Z^d)$ and $h>0$, and for any set of \textit{Schrödinger admissible} pairs
\[ \frac{2}{q}+ \frac{d}{r} = \frac{d}{2}, \quad 2\leq q,r \leq \infty, \quad (q,r,d) \neq (2,\infty,2).   \]
Such kind of strategy was also used by Killip Ouyang Visan and Wu \cite{killip2023} in the study of the continuum limit of the Ablowitz-Ladik model, where frequency-localized Strichartz estimates \cite[Proposition 4.3]{killip2023} were also derived in order to infer compactness on low-regularity discrete solutions.

Another approach was recently considered by Hong and Yang in \cite{hong2019strichartz}, where the authors showed that such $h$-dependence in equation \eqref{eq_decay_Schrödinger} can be removed paying fractional derivatives on the right hand side of Strichartz estimates, which compensates the lattice resonances. The proof relies on harmonic analysis tools such as Littlewood-Paley inequality, Calderon-Zygmund theory and the Hörmander-Mikhlin theorem adapted on the discrete setting. This implies a uniform Strichartz estimate of the form
\begin{equation} \label{eq_discrete_Strichartz_Schrödinger}
 \| e^{i t \Delta_h} \varphi \|_{L^q_t(\R;L^r_h)} \lesssim \| \varphi \|_{H^{\frac{1}{q}}_h}  
 \end{equation}
 for any set of \textit{discrete Schrödinger-admissible} pairs $(q,r)$ satisfying
\begin{equation} \label{eq_admissible_discrete}
\frac{3}{q}+ \frac{d}{r} = \frac{d}{2}, \quad 2\leq q,r \leq \infty, \quad (q,r,d) \neq (2,\infty,3). 
\end{equation}   
Note that this kind of discrete Strichartz estimates with loss were also recently proved by Choi and Aceves \cite{choi2023} for fractional-type discrete Schrödinger equations in the two-dimensional case, and are of course reminiscent of the situation occurring on compact manifolds like in the seminal work of Burq Gérard and Tzevtkov \cite{burq2004}.

To the best of the author's knowledge, no such techniques have yet been used in order to derive uniform Strichartz estimates for the multidimensional discrete Klein-Gordon equation. In the following we elaborate on this point. Making the change of variables 
\begin{equation} \label{eq_change_variables}
\xi \mapsto \frac{\xi'}{h}, \quad \tau= \frac{t}{h}, \quad v = \frac{a}{h \tau} 
\end{equation}   
in \eqref{eq_oscillatory_integral}, one is reduce to estimate the oscillatory integral
\[ J_{\Phi_v,\zeta} \coloneqq  \int_{\R^d} e^{i \tau \Phi_{v}(\xi)} \zeta(\xi) \dd \xi  \]
for any $v \in \R^d$ in the limit $\tau \rightarrow \infty$, where $\zeta \in \mathcal{C}_c^{\infty}(\R^d)$ and $\Phi_{v}(\xi)=v \cdot \xi - \gamma_h(\xi)$ with
\[ \gamma_h(\xi)=\sqrt{h^2+ 2\sum_{j=1}^d \cos (\xi_j) } \]
 denoting the dispersion relation. Note that in the two-dimensional case, on the lattice $\Z^2$ (with $h=1$), the work \cite{borovyk2017} provide a decay rate in $\tau^{-\frac{3}{4}}$ for $|J_{\Phi_v,\zeta}|$. However, at the limit $h \rightarrow 0$ we observe that $\gamma_h$ tends to the dispersion relation of the discrete wave equation, which would suggest a slower decay rate in~$\tau^{-\frac{2}{3}}$ of the oscillatory integral in view of \cite{schultz1998}. Comments on this pathological behavior are also given in \cite[Section 2.2]{borovyk2017}, but not from a continuum limit perspective, and we suggest the following conjecture:

\begin{conjecture} \label{prop_oscillatory_integral}
For any $\zeta \in \mathcal{C}_c^{\infty}(\R^d)$,
 \[ \sup_{v \in \R^d} | J_{\Phi_v,\zeta}| \leq C(\zeta) (1+|\tau|)^{-\frac{2}{3}}.  \]
\end{conjecture}
 
Assuming Conjecture \ref{prop_oscillatory_integral}, one can derive uniform Strichartz estimates for \eqref{DNLKG} for $d=2$ the following way. Let $\psi : \R^d \rightarrow \left[0,1\right]$ be a smooth even compactly supported function such that $\psi=1$ for $\xi \in \left[-\pi,\pi \right]^d$ and $\psi=0$ for $\xi \in \R^d \backslash \left[-2\pi,2\pi \right]^d$. Let $\eta(\xi) \coloneqq \psi(|\xi|)-\psi(2 |\xi|)$. We then define for dyadic integers $N\in 2^{\Z}$ the Littlewood-Paley projections given by 
\[ P_N \coloneqq \mathcal{F}^{-1} \eta \left(\frac{h\xi}{N} \right) \mathcal{F}.  \]
Since $\xi \in \T^d_h$, $P_N$ is a smooth projector onto the set $\frac{\pi}{2} \frac{N}{h} \leq |\xi| \leq 2\pi \frac{N}{h}$, and the $(P_N)_N$ resolves the identity 
\[ \sum_{N\leq 1} P_N = \Id. \]
 Note that the sum has an upper bound as $h\xi =\mathcal{O}_d(1)$. Denoting 
\[ K_{t,N,h}(a) \coloneqq I \left(t,a,\eta(h \cdot/N) \right),\]
 we infer that
\[ \| e^{-it \sqrt{1-\Delta_h}} P_N \varphi \|_{L^{\infty}_h} = \|K_{t,N,h} \ast \varphi  \|_{L^{\infty}_h} \leq  \|K_{t,N,h} \|_{L^{\infty}_h} \| \varphi \|_{L^1_h} \]
by Young's inequality. Making the change of variable \eqref{eq_change_variables} we get that $K_{t,N,h}(a)=h^{-d} J_{\Phi_v,\eta(\cdot/N)}$, so from Conjecture \ref{prop_oscillatory_integral} we write that
\[ \|K_{t,N,h} \|_{L^{\infty}_h} \leq \frac{C(\eta(\cdot/N))}{h^d} |\tau|^{-\frac{2}{3}} \leq C \left( \frac{N}{h} \right)^{d-\frac{2}{3}} |t|^{-\frac{2}{3}},  \]
which implies as $d=2$ that
\[ \| e^{-it \sqrt{1-\Delta_h}} P_N \varphi \|_{L^{\infty}_h} \leq C \left( \frac{N}{h} \right)^{\frac{4}{3}} |t|^{-\frac{2}{3}} \| \varphi \|_{L^1_h}   \]
for any $\varphi \in L^1(h\Z^2)$ and $N \leq 1$. Once such a bound has been obtained, it is then straightforward to derive the Strichartz estimates for the linear evolution  by averaging in $t$ and $N$ ,see for instance \cite[Remark 3.4]{choi2023}. In fact, denoting $U_h(t)=e^{-it \sqrt{1-\Delta_h}}$ and defining as in \cite[p.1127]{cho2011}
\[ \widetilde{U}_N(t)=P_N U_h\left(  \frac{N^2t}{h^2}\right) P_{\widetilde N}\]
where $P_{\widetilde N} =P_{N/2}+P_N+P_{2N}$, we can show that $(\widetilde{U}_N(t))_{t \in \R}$ satisfies the hypothesis of \cite[Theorem 1.2]{tao1998} which implies that
\[ \| U_h(t) P_N \varphi \|_{L^q_t(\R;L^r_h)} \lesssim \left( \frac{N}{h}  \right)^{\frac{4}{3}\left( \frac12 - \frac{1}{r} \right)} \| P_N \varphi \|_{L^2(h\Z^2)} \simeq \| P_N |\nabla_h |^{\frac{4}{3}\left( \frac12 - \frac{1}{r} \right)} \varphi\|_{L^2(h\Z^2)}.  \]
Squaring on both side and summing over $N \leq 1$, this gives that
\begin{align*}
\| U_h(t) \varphi \|_{L^q_t(\R;L^r_h)}  & \lesssim \left( \sum_{N \leq 1} \| U_h(t) P_N \varphi \|_{L^q_t(\R;L^r_h)}^2  \right)^{\frac12} \lesssim \left( \sum_{N\leq 1} \| P_N |\nabla_h|^{\frac{4}{3}\left( \frac12 - \frac{1}{r} \right)} \varphi \|_{L^2_h}^2   \right)^{\frac12} \\
& \lesssim \| |\nabla_h|^{\frac{4}{3}\left( \frac12 - \frac{1}{r} \right)} \varphi \|_{L^2_h},
\end{align*} 
which would finally provides in our case a uniform Strichartz estimate with loss of the form
\begin{equation} \label{eq_strichartz_DNLKG}
\| e^{-it \sqrt{1-\Delta_h}} \varphi \|_{L^q_t(\R;L^r(h\Z^2))} \leq C(q,r) \| \varphi \|_{H^{\frac{2}{q}}(h\Z^2)}  
\end{equation}  
in view of the admissible condition \eqref{eq_admissible_discrete}. Note that this estimate still give a gain of regularity compared to the trivial estimate induced by the Sobolev embedding $H^{\alpha}(h\Z^2) \subset L^r(h\Z^2)$ with $\alpha=1-2/r$, which would imply a $H^{\frac{3}{q}}(h\Z^d)$-norm in the right hand side of estimate~\eqref{eq_strichartz_DNLKG} in view of~\eqref{eq_admissible_discrete}.
 
To conclude this monograph, we point out that uniform discrete Strichartz with loss can also be derived on compact sets, which is in agreement with effective numerical simulations. Adapting the work of Vega \cite{vega1992} on the discrete setting, these kind of estimates were derived for the discrete Schrödinger flow on the discrete torus $\T_h^2$ in \cite{hong2021} or in a large box limit in dimension $d=2,3$ in \cite{hong2023}, and very recently for the discrete fractional Schrödinger equation on $\T_h$ in \cite{choi2024}. Note that contrary to those on $h \Z^d$ \cite{hong2019strong}, these discrete Strichartz on compact sets are not proved to be optimal or not yet.

\subsection*{Acknowledgements}
The author is supported by the Labex CEMPI (ANR-11-LABX-0007-01).

\bibliographystyle{siam}
\bibliography{biblio}

\begin{thebibliography}{10}

\bibitem{ablowitz1975}
{\sc M.~J. Ablowitz and J.~F. Ladik}, {\em Nonlinear differential-difference
  equations}, J. Mathematical Phys., 16 (1975), pp.~598--603.

\bibitem{audiard2013}
{\sc C.~Audiard}, {\em Dispersive schemes for the critical {K}orteweg-de
  {V}ries equation}, Math. Models Methods Appl. Sci., 23 (2013),
  pp.~2603--2646.

\bibitem{cheng2023}
{\sc C.~Bi, J.~Cheng, and B.~Hua}, {\em The wave equation on lattices and
  oscillatory integrals}, 2023.
\newblock Preprint, archived at \url{https://arxiv.org/abs/2312.04130v2}.

\bibitem{borovyk2017}
{\sc V.~Borovyk and M.~Goldberg}, {\em The {K}lein-{G}ordon equation on
  {$\Bbb{Z}^2$} and the quantum harmonic lattice}, J. Math. Pures Appl. (9),
  107 (2017), pp.~667--696.

\bibitem{burq2004}
{\sc N.~Burq, P.~G\'{e}rard, and N.~Tzvetkov}, {\em Strichartz inequalities and
  the nonlinear {S}chr\"{o}dinger equation on compact manifolds}, Amer. J.
  Math., 126 (2004), pp.~569--605.

\bibitem{chauleur2023dnls}
{\sc Q.~Chauleur}, {\em Growth of {S}obolev norms and strong convergence for
  the discrete nonlinear {S}chrödinger equation}, Nonlinear Analysis, 242
  (2024), p.~113517.

\bibitem{cho2011}
{\sc Y.~Cho, T.~Ozawa, and S.~Xia}, {\em Remarks on some dispersive estimates},
  Commun. Pure Appl. Anal., 10 (2011), pp.~1121--1128.

\bibitem{choi2024}
{\sc B.~Choi}, {\em Periodic fractional discrete nonlinear {S}chr\"odinger
  equation and modulational instability}, 2024.
\newblock Preprint, archived at \url{https://arxiv.org/abs/2401.13152}.

\bibitem{choi2023}
{\sc B.~Choi and A.~Aceves}, {\em Continuum limit of 2{D} fractional nonlinear
  {S}chr\"{o}dinger equation}, J. Evol. Equ., 23 (2023), pp.~Paper No. 30, 35.

\bibitem{cuenin2021}
{\sc J.-C. Cuenin and I.~A. Ikromov}, {\em Sharp time decay estimates for the
  discrete {K}lein-{G}ordon equation}, Nonlinearity, 34 (2021), pp.~7938--7962.

\bibitem{FPU1955}
{\sc E.~Fermi, J.~Pasta, and S.~Ulam}, {\em Studies of nonlinear problems.
  {I}}.
\newblock Nonlin. {Wave} {Motion}, {Proc}. {Summer} {Sem}. {Potsdam} ({New}
  {York}) 1972, 143-156 (1974)., 1974.

\bibitem{ginibre1985}
{\sc J.~Ginibre and G.~Velo}, {\em The global {C}auchy problem for the
  nonlinear {K}lein-{G}ordon equation}, Math. Z., 189 (1985), pp.~487--505.

\bibitem{ginibre1989}
\leavevmode\vrule height 2pt depth -1.6pt width 23pt, {\em The global {C}auchy
  problem for the nonlinear {K}lein-{G}ordon equation. {II}}, Ann. Inst. H.
  Poincar\'{e} Anal. Non Lin\'{e}aire, 6 (1989), pp.~15--35.

\bibitem{hong2021}
{\sc Y.~Hong, C.~Kwak, S.~Nakamura, and C.~Yang}, {\em Finite difference scheme
  for two-dimensional periodic nonlinear {S}chr\"{o}dinger equations}, J. Evol.
  Equ., 21 (2021), pp.~391--418.

\bibitem{hong2023}
{\sc Y.~Hong, C.~Kwak, and C.~Yang}, {\em On the continuum limit for the
  discrete nonlinear {S}chr\"{o}dinger equation on a large finite cubic
  lattice}, Nonlinear Anal., 227 (2023), pp.~Paper No. 113171, 26.

\bibitem{hong2019strong}
{\sc Y.~Hong and C.~Yang}, {\em Strong convergence for discrete nonlinear
  {S}chr\"{o}dinger equations in the continuum limit}, SIAM J. Math. Anal., 51
  (2019), pp.~1297--1320.

\bibitem{hong2019strichartz}
\leavevmode\vrule height 2pt depth -1.6pt width 23pt, {\em Uniform {S}trichartz
  estimates on the lattice}, Discrete Contin. Dyn. Syst., 39 (2019),
  pp.~3239--3264.

\bibitem{ignat2007}
{\sc L.~I. Ignat}, {\em Fully discrete schemes for the {S}chr\"{o}dinger
  equation. {D}ispersive properties}, Math. Models Methods Appl. Sci., 17
  (2007), pp.~567--591.

\bibitem{ignat2009}
{\sc L.~I. Ignat and E.~Zuazua}, {\em Numerical dispersive schemes for the
  nonlinear {S}chr\"{o}dinger equation}, SIAM J. Numer. Anal., 47 (2009),
  pp.~1366--1390.

\bibitem{ignat2012}
\leavevmode\vrule height 2pt depth -1.6pt width 23pt, {\em Convergence rates
  for dispersive approximation schemes to nonlinear {S}chr\"{o}dinger
  equations}, J. Math. Pures Appl. (9), 98 (2012), pp.~479--517.

\bibitem{tao1998}
{\sc M.~Keel and T.~Tao}, {\em Endpoint {S}trichartz estimates}, Amer. J.
  Math., 120 (1998), pp.~955--980.

\bibitem{kevrekidis2011}
{\sc P.~G. Kevrekidis}, {\em Non-linear waves in lattices: past, present,
  future}, IMA J. Appl. Math., 76 (2011), pp.~389--423.

\bibitem{killip2023}
{\sc R.~Killip, Z.~Ouyang, M.~Visan, and L.~Wu}, {\em Continuum limit for the
  {A}blowitz-{L}adik system}, Nonlinearity, 36 (2023), pp.~3751--3775.

\bibitem{staffilani2013}
{\sc K.~Kirkpatrick, E.~Lenzmann, and G.~Staffilani}, {\em On the continuum
  limit for discrete {NLS} with long-range lattice interactions}, Comm. Math.
  Phys., 317 (2013), pp.~563--591.

\bibitem{machihara2002}
{\sc S.~Machihara, K.~Nakanishi, and T.~Ozawa}, {\em Nonrelativistic limit in
  the energy space for nonlinear {K}lein-{G}ordon equations}, Math. Ann., 322
  (2002), pp.~603--621.

\bibitem{masmoudi2002}
{\sc N.~Masmoudi and K.~Nakanishi}, {\em From nonlinear {K}lein-{G}ordon
  equation to a system of coupled nonlinear {S}chr\"{o}dinger equations}, Math.
  Ann., 324 (2002), pp.~359--389.

\bibitem{mauser2020}
{\sc N.~J. Mauser, Y.~Zhang, and X.~Zhao}, {\em On the rotating nonlinear
  {K}lein-{G}ordon equation: nonrelativistic limit and numerical methods},
  Multiscale Model. Simul., 18 (2020), pp.~999--1024.

\bibitem{megias2022}
{\sc E.~Megias, M.~J. Teixeira, V.~S. Timoteo, and A.~Deppman}, {\em Nonlinear
  {K}lein-{G}ordon equation and the {B}ose-{E}instein condensation}, European
  {P}hysical {J}ournal {P}lus, 137 (2022).

\bibitem{pampu2021}
{\sc A.~B. Pampu}, {\em On the growth of the {S}obolev norms for the nonlinear
  {K}lein-{G}ordon equation}, J. Dynam. Differential Equations, 33 (2021),
  pp.~817--832.

\bibitem{peyrard1989}
{\sc M.~Peyrard and A.~R. Bishop}, {\em Statistical mechanics of a nonlinear
  model for {D}{N}{A} denaturation}, Phys. Rev. Lett., 62 (1989),
  pp.~2755--2758.

\bibitem{planchon2017}
{\sc F.~Planchon, N.~Tzvetkov, and N.~Visciglia}, {\em On the growth of
  {S}obolev norms for {NLS} on 2- and 3-dimensional manifolds}, Anal. PDE, 10
  (2017), pp.~1123--1147.

\bibitem{schultz1998}
{\sc P.~Schultz}, {\em The wave equation on the lattice in two and three
  dimensions}, Comm. Pure Appl. Math., 51 (1998), pp.~663--695.

\bibitem{kevrekidis2005}
{\sc A.~Stefanov and P.~G. Kevrekidis}, {\em Asymptotic behaviour of small
  solutions for the discrete nonlinear {S}chr\"{o}dinger and {K}lein-{G}ordon
  equations}, Nonlinearity, 18 (2005), pp.~1841--1857.

\bibitem{toda1989}
{\sc M.~Toda}, {\em Theory of nonlinear lattices}, vol.~20 of Springer Series
  in Solid-State Sciences, Springer-Verlag, Berlin, second~ed., 1989.

\bibitem{ustinov1993}
{\sc A.~V. Ustinov, M.~Cirillo, and B.~A. Malomed}, {\em Fluxon dynamics in
  one-dimensional {J}osephson-junction arrays}, Phys. Rev. B, 47 (1993),
  pp.~8357--8360.

\bibitem{vega1992}
{\sc L.~Vega}, {\em Restriction theorems and the {S}chr\"{o}dinger multiplier
  on the torus}, in Partial differential equations with minimal smoothness and
  applications ({C}hicago, {IL}, 1990), vol.~42 of IMA Vol. Math. Appl.,
  Springer, New York, 1992, pp.~199--211.

\end{thebibliography}

\end{document}